\documentclass[10pt]{amsart}
\usepackage{amsmath,amsthm,amssymb,amsfonts,eucal,amscd,mathbbol,mathrsfs,tikz-cd}
\usepackage[shortlabels]{enumitem}
\usepackage{relsize}
\usepackage{cite}
\usepackage{setspace}
\usepackage[all,cmtip]{xy}
\usepackage{graphicx}
\usepackage{subfig}
\usepackage{leftidx}
\usepackage[bookmarks,bookmarksnumbered, plainpages=false, pdfpagelabels]{hyperref}
\usepackage{xcolor}
\usepackage{extarrows}
\usepackage{rotating}
\usepackage{scalefnt}
\usepackage{enumitem}

\hypersetup{
 colorlinks   = true, 
 urlcolor     = green, 
 linkcolor    = blue, 
 citecolor   = red 
}

\newtheorem{thm}{Theorem}

\newtheorem{lem}{Lemma}
\newtheorem{prop}{Proposition}

\theoremstyle{definition}
\newtheorem{defn}{Definition}

\newtheorem{remark}{Remark}

\newcommand{\R}{\ensuremath{\mathbb{R}}}

\def\i{\infty}

\def\p{\partial}
\def\cal{\mathcal}

\def\Ad{\mathrm{Ad}}
\def\A{\mathrm{A}}
\def\sgn{\mathrm{sgn}}
\def\sh{\mathrm{Sh}}

\def\a{\alpha}
\def\b{\beta} 
\def\g{\gamma}

\def\e{\epsilon}

\def\s{\sigma} 

\def\o{\omega}

\def\G{\Gamma}
\def\D{\Delta}

\def\H{\cal{H}}

\def\L{\cal{L}}

\def\U{\cal{U}}

\def\fg{\mathfrak{g}}
\def\fh{\mathfrak{h}}

\renewcommand{\theenumi}{(\alph{enumi})}
\makeatletter
	\renewcommand{\p@enumii}{\theenumi}
\makeatother                                            

\begin{document} 
	
\author[H. Pugh]{H. Pugh} 
\title{The Hopf Algebraic Structure of Finitely Supported Currents on a Lie Group}

\maketitle
\begin{abstract}
The space of de Rham currents supported in finitely many points in a Lie group \( G \) has the structure of a filtered differential graded Hopf algebra. The product is given by convolution of compactly supported currents, and the co-product dualizes to wedge product on differential forms. This space arises as the finitely supported sections functor \( \G^{finite} \) applied to the bundle \( \U(G) \) of currents on \( G \) supported at a single (variable) point, and the differential Hopf algebra operations pull back via \( \G^{finite} \) to bundle maps. Explicit formulas for these bundle maps are obtained, and we show in particular that the convolution product takes the form of a Hopf-algebraic smash product.
\end{abstract}

\section{Introduction}
Let \( G \) be a Lie group with Lie algebra \( \fg \) and let \( \U(G)=\oplus_k \U_k(G) \) be the vector bundle of de Rham currents on \( G \) supported in a single (variable) point (see \cite{coprodman},) graded by dimension \( k \) and filtered by order\footnote{A compactly supported current \( T \) has order \( r \) if it extends to a continuous linear functional on \( C^r \) forms, and does not extend further to \( C^s \) forms for \( s<r \). Every compactly supported current has finite order.} \( r \), \( \U_k(G)=\cup_r \U_k^r(G) \). The fiber \( \U(G;p) \) of \( \U(G) \) at \( p \) consists of currents whose support is contained in \( \{p\} \), and has the structure of a filtered differential graded co-algebra: The co-algebra \( (\D_p, \e_p) \) dualizes to the exterior algebra of differential forms, and the boundary \( \p_p \) is dual to exterior derivative. 

The boundary \( (\p_p)_{p\in G} \), co-product \( (\D_p)_{p\in G} \) and co-unit \( (\e_p)_{p\in G} \) maps stitch together to form smooth bundle maps \( \tilde{\p}: \U_k^r(G) \to \U_{k-1}^{r+1}(G) \), \( \tilde{\D}: \U_k^r(G)\to \oplus_{i+j=k}\sum_{s+t=r}\U_i^s(G)\otimes \U_j^t(G) \), and \( \tilde{\e}: \U(G)\to G\times \R \), respectively, covering the identity on \( G \). 

The convolution of two compactly supported currents in a Lie group is well-defined \cite{guillemot}, \cite{schwartz0} and restricts to a map
\[
\ast_{p,q}: \U_i^s(G;p) \otimes \U_j^t(G;q) \to \U_{i+j}^{s+t}(G;pq).
\]

These convolution product maps \( (\ast_{p,q})_{(p,q)\in G\times G} \) stitch together to form a smooth bundle map \( \tilde{\ast} \) whose domain is the bundle \( \pi_1^* \U_i^s(G) \otimes \pi_2^* \U_j^t(G) \) over \( G\times G \) (the maps \( \pi_i:G\times G\to G \), \( i=1,2 \) being the projections into the first and second factor, respectively,) with fiber \( \U_i^s(G;p)\otimes \U_j^t(G;q) \) at \( (p,q) \), and whose range is \( \U_{i+j}^{s+t}(G) \). The bundle map \( \tilde{\ast} \) covers the group law \( G\times G\to G \).

Applying the finitely supported sections functor \( \G^{finite} \) to the maps \( \tilde{\p} \), \( \tilde{\D} \), \( \tilde{\e} \), and \( \tilde{\ast} \), we get, respectively, a differential \( \p \), co-product \( \D \), co-unit \( \e \), and product \( \ast \) on the space \( \G^{finite}(\U(G)) \). Moreover, identifying \( \G^{finite}(\U(G)) \) with the space of finitely supported currents on \( G \),
\begin{thm}
	\label{thm:hopf1}
	The space of finitely supported currents on \( G \) has the structure of a filtered differential graded Hopf algebra. The Hopf algebra maps pull back via \( \G^{finite} \) to bundle maps. In particular, the antipode pulls back to a bundle map on \( \U(G) \) covering the inversion map \( g\mapsto g^{-1} \) on \( G \).
\end{thm}
In particular, \( \p \) commutes with both the product and co-product structures, the antipode \( S \) satisfies \( S\p = (-1)^{k-1} \p S \) on \( \G^{finite}(\U_k(G)) \) and \( S(\G^{finite}(\U_k^r(G)))\subseteq \G^{finite}(\U_k^r(G)) \).

We now provide explicit formulas for these Hopf algebra maps. The bundle \( \U_k(G) \) is naturally isomorphic to the (trivial) bundle \[ G\times \left(U(\fg)\otimes \wedge^k(\fg^-)\right), \] where \( U(\fg) \) is the universal enveloping algebra of the Lie algebra \( \fg \) of \( G \) and \( \fg^- \) is the Lie algebra \( \fg \) with negative Lie bracket. Here we are interpreting the elements of \( U(\fg) \) as left-invariant differential operators \( C^\i(G)\to C^\i(G) \), and \( \wedge^k(\fg^-) \) as right-invariant tangent \( k \)-vector fields, so \( (p, v\otimes \a)\in  G\times U(\fg)\otimes \wedge^k(\fg^-) \) is identified with the current \( \o \mapsto [v(\o)]_p(\a), \) where \( v \) is extended to \( k \)-forms by the Leibniz rule and commutation with \( d \).

Transporting the bundle maps \( \tilde{\p} \), \( \tilde{\D} \), \( \tilde{\e} \) and \( \tilde{\ast} \) to \( G\times (U(\fg)\otimes \wedge(\fg^-)) \), we have the following formulas:

For \( v\in U(\fg) \) and \( \a=\a_0\wedge\cdots\wedge \a_k\in \wedge^{k+1}(\fg^-) \), the boundary \( \tilde{\p}(p, v\otimes \a) \) is equal to
\begin{equation}
	\label{eq:boundary}
	 \left(p,\sum_{i=0}^k(-1)^i (\Ad_p^{-1}(\a_i)) v \otimes \a_0\wedge\cdots \hat{\a_i}\cdots\wedge \a_k + \sum_{i<j}(-1)^{i+j+1}v\otimes [\a_i,\a_j]\wedge\a_1\wedge\cdots \hat{\a_i}\cdots\hat{\a_j}\cdots\wedge\a_k\right).
\end{equation}
The co-product bundle map \( \tilde{\D} \) is given, using Sweedler notation for the co-products in the Hopf algebras \( U(\fg) \) and \( \wedge(\fg^-) \), by
\begin{equation}
	\label{eq:coproduct}
	\tilde{\D}(p, v\otimes \a)=(p, (v_{(1)}\otimes \a_{(1)}) \otimes (v_{(2)}\otimes \a_{(2)})),
\end{equation}
and the co-unit bundle map is given by
\begin{equation}
	\label{eq:counit}
	\tilde{\e}(p, v\otimes \a)=(p, \e(v)\e(\a)),
\end{equation}
where \( \e(v) \) and \( \e(\a) \) refer to the co-units of \( v \) and \( \a \) in \( U(\fg) \) and \( \wedge(\fg^-) \), respectively. In other words, \( (\tilde{\D}, \tilde{\e}) \) is constant across fibers, equal to the tensor product of the co-algebras \( U(\fg) \) and \( \wedge(\fg^-) \).


Equation \eqref{eq:boundary} follows from the invariant formula for exterior derivative, since the right-invariant vector field equal to \( \a_i \) at \( e \) and the left invariant vector field equal to \( (\Ad_p^{-1}(\a_i)) \) at \( e \) coincide at \( p \).  Equation \eqref{eq:coproduct} follows from the Leibniz rule for iterated Lie derivatives\footnote{For vector fields \( v_1,\dots, v_r \), \[ \L_{v_1}\cdots \L_{v_r} (\o \wedge \eta)= \sum_{s=0}^r \sum_{\s\in Sh(s,r-s)}\left(\L_{v_{\sigma(1)}}\cdots\L_{v_{\sigma(s)}} \o \right)\wedge \left(\L_{v_{\sigma(s+1)}}\cdots\L_{v_{\sigma(r)}} \eta\right), \] where \( Sh(s, r-s) \) is the set of \( (s,r-s) \)-shuffle permutations.} and the duality between the co-product and exterior product on \( \wedge(\fg^-) \) and \( \wedge(\fg^-)^* \), respectfully. 
The Leibniz rule for exterior derivative then implies the commutation of \( \tilde{\p} \) and \( \tilde{\D} \) (this can also be checked from the formulas, as can the property that \( \tilde{\p}^2=0 \).)

The bundle map \( \tilde{\ast} \) is given, for \( (p, v\otimes \a)\in \U_k(G;p) \) and \( (q, w\otimes \b)\in \U_j(G;q) \) by 
\begin{equation}
	\label{eq:convolution}
	\tilde{\ast}((p,q), (v\otimes \a)\otimes(w\otimes \b)) = (pq, \Ad_q^{-1}(v_{(1)}) w\otimes \a\wedge \Ad_p[v_{(2)},\b]).
\end{equation}

The notation \( [\cdot,\cdot]: U(\fg)\times \wedge(\fg^-)\to \wedge(\fg^-) \) refers to the extension of the Lie bracket on \( \fg \) using the universal properties of \( U(\fg) \) and \( \wedge(\fg^-) \). The adjoint action \( \Ad \) is extended similarly to \( U(\fg) \) and \( \wedge(\fg^-) \). We derive \eqref{eq:convolution} in Lemma \ref{lem:conv}.

\begin{remark}
	The space \( \G^{finite}(\U(G)) \) is trivially isomorphic to the group ring \( (U(\fg)\otimes \wedge(\fg^-))[G] \), which is equipped with a convolution product \( \ast' \). Explicitly, \[ (p, v\otimes \a)\ast' (q, w\otimes \b) = (pq, (vw)\otimes (\a\wedge \b)). \] However, this convolution product does not agree with the convolution product \( \ast \) of compactly supported currents as given in \eqref{eq:convolution}, so we can think of \( \G^{finite}(\U(G)) \) as a quantization of \( (U(\fg)\otimes \wedge(\fg^-))[G] \).
\end{remark}

\begin{remark}
	The convolution product formula \eqref{eq:convolution} looks suspiciously like a Hopf-algebraic smash product (see \cite{sweedler}). Indeed, it (almost) is: The space \( \G^{finite}(\U(G)) \) is naturally isomorphic to \( A\otimes_{C^\i(G)}H \), where \( A=\G^{finite}(G\times \wedge(\fg^-)) \) and \( H=\G^{finite}(G\otimes U(\fg)) \). The space \( H \) is a Hopf algebra, being a Hopf sub-algebra of \( \G^{finite}(\U(G)) \), and \( H \) acts on \( A \) as follows: For \( (p,v)\in H \) and \( (q,\a)\in A \), \[ (p,v)\cdot (q,\a) := (pq, Ad_p[v,\a]). \] This turns \( A \) into a (left) Hopf \( H \)-module algebra, and we can define the (Hopf-algebraic) smash product on \( A\otimes_\R H \). This product is also well-defined on \( A\otimes_{C^\i(M)}H \), hence on \( \G^{finite}(\U(G)) \), and is none other than the convolution product \( \ast \).
\end{remark} 

Finally, letting \( S(v) \) and \( S(\a) \) be the antipodes of \( v\in U(\fg) \) and \( \a\in \wedge(\fg^-) \), respectively, the fiber-wise antipode map \( S_p: \U_k(G;p)\to \U_k(G;p^{-1}) \) given by
\begin{equation}
	\label{eq:antipode}
	S_p(p, v\otimes \a)=(p^{-1}, \Ad_p(S(v_{(1)}))\otimes [S(v_{(2)}),\Ad_p^{-1}S(\a)]),
\end{equation}
which defines a smooth bundle map \( \tilde{S}: \U(G)\to \U(G) \) covering the inversion \( g\mapsto g^{-1} \) map on \( G \). We will show that \( \tilde{S} \) satisfies the equations required of an antipode, and that it commutes (up to sign) with \( \tilde{\p} \).

Using the natural quotient \( \otimes(\fg)\to U(\fg) \), there is a natural quotient of bundles \( \Phi: \otimes(TG)\otimes \wedge(TG)\to \U(G) \). The bundle maps \( \tilde{\p}, \tilde{\ast}, \tilde{\D} \), \( \tilde{\e} \) and \( \tilde{S} \) lift via \( \Phi \) to smooth bundle maps on \( \otimes(TG)\otimes \wedge(TG)\simeq G\times \otimes(\fg)\otimes \wedge(\fg^-) \), and 

\begin{thm}
	\label{thm:hopf3}
	Applying \( \G^{finite} \) to these lifted maps, the space \( \G^{finite}(G\times \otimes(\fg)\otimes \wedge(\fg^-)) \) has the structure of a filtered and graded Hopf algebra. Moreover, the lift of \( \p \), while not a differential, respects the filtration and grading, and commutes with the co-algebra structure as a differential would.
\end{thm}

The differential graded Hopf algebra structure on \( \G^{finite}(\U(G)) \) is a special case of a more general set-up: 

\begin{thm}\label{thm:hopf2}
	If \( \A: H\to Aut(\fh) \) is a representation of a group \( H \) on a Lie algebra \( \fh \), then the formulas \eqref{eq:boundary}-\eqref{eq:antipode}, replacing the adjoint representation with \( \A \), define a filtered differential graded Hopf algebra structure on \( (U(\fh)\otimes \wedge(\fh^-))[H] \). 
\end{thm}

It follows from \eqref{eq:boundary}-\eqref{eq:antipode} that if \( H' \) is a subgroup of \( H \), then \( (U(\fh)\otimes \wedge(\fh^-))[H'] \) is a differential Hopf sub-algebra of \( (U(\fh)\otimes \wedge(\fh^-))[H] \). In particular, if \( e\in H \) is the identity element, then \( (U(\fh)\otimes \wedge(\fh^-))[\{e\}] \) is a differential Hopf sub-algebra of \( (U(\fh)\otimes \wedge(\fh^-))[H] \), and contains \( U(\fg) \) and \( \wedge(\fh^-) \) as Hopf sub-algebras: the former as the \( 0 \)-graded component and the latter as the \( 0 \)-filtered component.

\begin{thm}\label{thm:hopf4}
	If \( H \) is topological (resp. Lie) and \( \A \) is continuous then the Hopf algebra maps from Theorem \ref{thm:hopf2} define continuous (resp. smooth) vector bundle homomorphisms which lift to \( H\times (\otimes(\fh)\otimes \wedge(\fh^-)) \), and Theorem \ref{thm:hopf3} holds in this context. 
\end{thm}

\subsection{Outline}
We will first prove that the convolution product of compactly supported currents has the formula given by \eqref{eq:convolution}. We will then prove Theorem \ref{thm:hopf2} directly using \eqref{eq:boundary}-\eqref{eq:antipode} as definitions for the Hopf algebra operations, and this together with the above will imply Theorem \ref{thm:hopf1}. We will then prove Theorem \ref{thm:hopf4}, which will imply Theorem \ref{thm:hopf3}.




\section{Convolution Product}
\begin{defn}\label{def:conv}
	Let \( A \) and \( B \) be compactly supported currents in \( G \). The convolution product \( A\ast B \) is the compactly supported current in \( G \) defined by  \( A\ast B (\o) = A\otimes B (m^* \o) \), where \( m^* \) is the pullback by the group multiplication \( m: G\times G\to G \) and \( A\otimes B \) is the tensor product of \( A \) and \( B \). The tensor product \( A\otimes B \) is the current on \( G\times G \) uniquely determined by the property \( A\otimes B (\pi_1^* \eta \wedge \pi_2^* \zeta)=A(\eta)B(\zeta) \) for forms \( \eta \) and \( \zeta \) on \( G \).
\end{defn}

Our goal for this section will be to derive \eqref{eq:convolution}. It will be useful to temporarily keep track of right- as well as left-invariant differential operators: we will use the notation \( [p, v\otimes \a \otimes u] \), where \( p\in G \), \( v\in U(\fg) \), \( \a\in \wedge(\fg^-) \) and \( u\in U(\fg^-) \), to denote the current \( \o \mapsto (v\circ u (\o))_p(\a) \), where \( u \) is interpreted as a right-invariant differential operator on \( C^\i(G) \). Note that as differential operators, \( v\circ u = u\circ v \), so the order of application does not matter. First, we derive the following special case for the convolution product:
\begin{lem}\label{lem:conv0}
	\[ [p, 1\otimes \a \otimes 1] \ast [q, 1\otimes \b \otimes 1]= [pq, 1\otimes \a \wedge \Ad_p\b \otimes 1]. \]
\end{lem}
\begin{proof}
	A multi-vector \( \a\in \wedge(T_pG) \) determines a current \( T_\a \), given by \( T_\a(\o)=\o_p(\a) \). The identification between \( \a \) and \( T_\a \) is natural in the sense that pushforward of currents and pushforward of multi-vectors agree. The current \( T_\a \otimes T_\b \) is determined by the multi-vector \( {\iota_1}_* \a \wedge {\iota_2}_* \b\in \wedge(T_{(p,q)}G\times G) \), where \( \iota_i: G\to G\times G \) is the inclusion into the \( i \)-th factor, \( i=1,2 \). The pushforward of \( {\iota_1}_* \a \wedge {\iota_2}_* \b \) by the group law is the multi-vector \( {R_q}_* \a \wedge {L_p}_* \b \in \wedge(T_{pq} G) \), and the result follows.
\end{proof}

This gives a formula for \( \ast \) on \( \U^0(G) \). Now we extend \( \ast \) to the rest of \( \U(G) \).
\begin{lem}\label{lem:conv1}
	\( [p, 1\otimes \a\otimes u] \ast [q, w\otimes \b\otimes 1] = [pq, w\otimes \a\wedge \Ad_p\b \otimes u]. \)
\end{lem}
\begin{proof}
	By linearity, it suffices to prove the result for \( u=u_1\circ\cdots\circ u_s \) and \( w=w_1\circ\cdots\circ w_t \), for right-invariant vector fields \( u_1,\dots,u_s \) and left-invariant vector fields \( w_1,\dots,w_t \).
	
	Note that \[ [p, w\otimes \a \otimes u_1\circ\cdots\circ u_s ] = \lim_{h\to 0} \frac{1}{h} \left([\exp(h u_1)p, w\otimes \Ad_{\exp(h u_1)} \a \otimes u_2\circ\cdots\circ u_s] - [p, w\otimes \a \otimes u_2\circ\cdots\circ u_s ]\right) \]
	and similarly \[ [q, w_1\circ\cdots\circ w_t \otimes \b \otimes u] = \lim_{h\to 0} \frac{1}{h} \left([q \exp(h w_1), w_2\circ\cdots\circ w_t\otimes \b \otimes u] - [q, w_2\circ\cdots\circ w_t\otimes \b \otimes u ]\right), \]
	where the limit is in the sense of weak convergence. The result then follows from separate weak continuity of \( \ast \) and induction on \( s \) and \( t \), using Lemma \ref{lem:conv0} for the \( s=t=0 \) case.
\end{proof}

The next lemma describes how to move between left- and right-invariant differential operators.
\begin{lem}\label{lem:conv2}
	\begin{align*} 
		[p, v\otimes \a \otimes u] &= [p, \Ad_p^{-1}(|S|^{-1}(u_{(1)}))v\otimes [|S|^{-1}(u_{(2)}), \a] \otimes 1]\\
		&= [p, 1\otimes [\Ad_p(S(v_{(1)})), \a] \otimes \Ad_p(|S|(v_{(2)}))u ],
	\end{align*}
	where \( |S|: U(\fg)\to U(\fg^-) \) is the algebra anti-isomorphism defined on \( \fg\subset U(\fg) \) as the map \( x\mapsto x \), and extended using the universal property for \( U(\fg) \). That is, \( |S| \) reverses the order of composition of left-invariant vector fields, but does so without the alternating sign present in the antipode map \( S: U(\fg)\to U(\fg) \). Note that \( |S| \) and \( S \) are co-algebra isomorphisms and commutes with \( \Ad \).
\end{lem}
\begin{proof}
	For \( x\in \wedge(T_e G) \), let \( x^L \) (resp. \( x^R \)) denote the left- (resp. right-)invariant vector field equal to \( x \) at \( e \).
	
	Given a form \( \o \), let \( \eta = v(\o) \). Then
	\begin{align*}
		[p, v\otimes \a \otimes u] (\o) &= (u \circ v (\o))_p(\a)\\
		&=(u(\eta))_p(\a)\\
		&=(u(\eta))_p({L_p}_* (\Ad_p^{-1}\a(e)))\\
		&=u(\eta((\Ad_p^{-1}\a(e))^L))_p\\
		&=\Ad_p^{-1}|S|^{-1}(u) (\eta((\Ad_p^{-1}\a(e))^L))_p\\
		&=(\Ad_p^{-1}(|S|^{-1}(u_{(1)}))\eta)_p([|S|^{-1}(u_{(2)}),\a]).
	\end{align*}
	The second-to-last line follows from the fact that left- and right-invariant differential operators commute, and the derivative of a function at \( p \) by a vector field \( X \) depends only on the value of \( X(p) \); thus the right-invariant vector fields constituting \( u \) can be switched to left-invariant vector fields by reversing their order (i.e. applying \( |S|^{-1} \))  and then applying \( \Ad_p^{-1} \). 
	
	Now let \( \mu=u(\o) \). Then
	\begin{align*}
		[p, v\otimes \a \otimes u] (\o) &= (u \circ v (\o))_p(\a)\\
		&=(v(\mu))_p(\a)\\
		&=v(\mu(\a))_p\\
		&=(\Ad_p |S|(v)) (\mu(\a))_p\\
		&=(\Ad_p |S|(v_{(1)})\mu)_p([\Ad_p S(v_{(2)}),\a]).
	\end{align*}
\end{proof}

We can now use Lemmas \ref{lem:conv1} and \ref{lem:conv2} to derive \eqref{eq:convolution}:
\begin{lem}\label{lem:conv}
	The convolution map \( \tilde{\ast}: (G\times G)\times (U(\fg)\otimes \wedge^k(\fg^-))\otimes (U(\fg)\otimes \wedge^k(\fg^-))\to G\times U(\fg)\otimes \wedge^k(\fg^-) \) is given by 
	\[ \tilde{\ast} ((p,q), (v\otimes \a)\otimes (w\otimes \b)) = (pq, \Ad_q^{-1}(v_{(1)}) w\otimes \a\wedge \Ad_p[v_{(2)},\b]). \]
\end{lem}

\begin{proof} 
	By Lemma \ref{lem:conv2}, the convolution \( (p, v\otimes \a) \ast (q, w\otimes \b) \) is equal to
	\[ [p, 1\otimes [\Ad_p(S(v_{(1)})), \a] \otimes \Ad_p(|S|(v_{(2)})) ] \ast [q, w\otimes \b\otimes 1], \]
	and by Lemma \ref{lem:conv1}, this is equal to
	\[ [pq, w\otimes [\Ad_p(S(v_{(1)})), \a] \wedge \Ad_p\b \otimes \Ad_p(|S|(v_{(2)}))]. \]
	Applying Lemma \ref{lem:conv2}, this becomes
	\[ [pq, \Ad_{pq}^{-1}(|S|^{-1}((\Ad_p(|S|(v_{(2)})))_{(1)}))w\otimes [|S|^{-1}((\Ad_p(|S|(v_{(2)})))_{(2)}), [\Ad_p(S(v_{(1)})), \a] \wedge \Ad_p\b ] \otimes 1], \]
	and simplifying, we get
	\begin{align*}
		&\quad\,\,[pq, \Ad_q^{-1}(v_{(1)}) w\otimes \Ad_p \left([v_{(2)}, [S(v_{(3)}), \Ad_p^{-1}\a] \wedge \b ]\right) \otimes 1] \\
		&=[pq, \Ad_q^{-1}(v_{(1)}) w\otimes \Ad_p \left([v_{(2)}, [S(v_{(3)}), \Ad_p^{-1}\a]] \wedge [v_{(4)} ,\b] \right) \otimes 1] \\
		&=[pq, \Ad_q^{-1}(v_{(1)}) w\otimes \Ad_p \left([v_{(2)}S(v_{(3)}), \Ad_p^{-1}\a] \wedge [v_{(4)} ,\b] \right) \otimes 1] \\
		&=(pq, \Ad_q^{-1}(v_{(1)}) w\otimes \a\wedge \Ad_p[v_{(2)},\b]).
	\end{align*}
\end{proof}

\section{Differential Graded Hopf Algebra}
We now verify that the space \( \H:= (U(\fh)\otimes \wedge(\fh^-))[H] \) is a differential graded Hopf algebra, where \( H \) is an arbitrary group and \( \fh \) is an arbitrary Lie algebra equipped with a Lie algebra representation \( \A: H\to Aut(\fh)\), the differential and Hopf algebra operations being defined in \eqref{eq:boundary}-\eqref{eq:antipode}, with \( \A \) taking the place of \( \Ad \). 

\begin{prop}\label{prop:assoc}
	The convolution product \( \ast \) on \( \H \) is associative.
\end{prop}
\begin{proof}
	On one hand, 
	\begin{align*}
		\left[(p, v\otimes \a)\ast(q, w\otimes \b)\right]&\ast(r, u\otimes \g)\\
		&= (pq, \A_q^{-1}(v_{(1)}) w\otimes \a\wedge \A_p[v_{(2)},\b]) \ast (r, u\otimes \g)\\
		&=(pqr, \A_r^{-1}(\A_q^{-1}(v_{(1)})w_{(1)})u\otimes \a\wedge \A_p[v_{(2)},\b]\wedge \A_{pq}[\A_q^{-1}(v_{(3)})w_{(2)},\g]).
	\end{align*}
	On the other hand,
	\begin{align*}
		(p, v\otimes \a)&\ast\left[(q, w\otimes \b)\ast(r, u\otimes \g)\right]\\
		&=(p, v\otimes \a) \ast(qr, \A_r^{-1}(w_{(1)})u\otimes \b\wedge \A_q [w_{(2)},\g])\\
		&=(pqr, \A_{qr}^{-1}(v_{(1)})\A_r^{-1}(w_{(1)})u\otimes \a\wedge \A_p [v_{(2)}, \b\wedge \A_q[w_{(2)},\g]]).
	\end{align*}
\end{proof}

There is also a unit for \( \ast \), namely \( (e, 1\otimes 1) \). 

\begin{prop}\label{prop:bialg}
	The diagram
	\[\begin{tikzcd}
		{H\otimes H} & H & {H\otimes H} \\
		{H\otimes H\otimes H \otimes H} && {H\otimes H\otimes H\otimes H}
		\arrow["\ast", from=1-1, to=1-2]
		\arrow["{\D \otimes \D}"', from=1-1, to=2-1]
		\arrow["\D", from=1-2, to=1-3]
		\arrow["{id\otimes \tau \otimes id}"', from=2-1, to=2-3]
		\arrow["{\ast \otimes \ast}"', from=2-3, to=1-3]
	\end{tikzcd}\]
	commutes, where \( \tau(x\otimes y)= (-1)^{|x||y|}y\otimes x \) for homogeneous elements \( x \) and \( y \).
\end{prop}
\begin{proof}
	On one hand, \( \D((p, v\otimes \a)\ast (q, w\otimes \b)) \) is equal to \[ (-1)^{|\a_{(2)}||\b_{(1)}|}(pq, \A_q^{-1}(v_{(1)})w_{(1)}\otimes \a_{(1)}\wedge \A_p[v_{(2)},\b_{(1)}])\otimes (pq, \A_q^{-1}(v_{(3)})w_{(2)}\otimes \a_{(2)}\wedge \A_p[v_{(4)},\b_{(2)}]), \]
	and on the other,
	\[ (-1)^{|\a_{(2)}||\b_{(1)}|}\left[(p, v_{(1)}\otimes \a_{(1)})\ast (q, w_{(1)}\otimes \b_{(1)})\right]\otimes \left[(p, v_{(2)}\otimes \a_{(2)})\ast (q, w_{(2)}\otimes \b_{(2)}) \right] \]
	is equal to the same, by co-associativity of \( \D \) on \( U(\fh) \).
\end{proof}

Additionally, we have 
\begin{align*}
	\e((p, v\otimes \a)\ast (q, w\otimes \b)) &= \e(pq, \A_q^{-1}(v_{(1)}) w\otimes \a\wedge \A_p[v_{(2)},\b])\\
	&= \e(\A_q^{-1}(v_{(1)}) w)\cdot \e(\a\wedge \A_p[v_{(2)},\b]),
\end{align*}
which is equal to zero unless \( \a \) and \( \b \) are both \( 0 \)-vectors, in which case, \( \A \) being extended to \( 0 \)-vectors by the identity, 
\[ 
\e((p, v\otimes \a)\ast (q, w\otimes \b)) = \e(\A_q^{-1}(v) w) \a\b.
\]
The right hand side is zero unless \( v \) and \( w \) are order-zero elements of \( U(\fh) \), i.e. elements of the ground field, in which case \( \e((p, v\otimes \a)\ast (q, w\otimes \b)) = vw\a\b \).

On the other hand, for the same reasons, \( \e(p, v\otimes \a) \cdot \e(q, w\otimes \b) \) equals zero unless \( \a \) and \( \b \) are \( 0 \)-vectors and \( v \) and \( w \) are order zero, in which case \[ \e(p, v\otimes \a) \cdot \e(q, w\otimes \b) = v\a w\b, \] showing that \( \e \otimes \e = \e\circ \ast \) on \( \H \).

Likewise, \( \D(e, 1\otimes 1)= (e, 1\otimes 1) \otimes (e, 1\otimes 1) \) and \( \e(e, 1\otimes 1)= 1 \). This, together with Propositions \ref{prop:assoc} and \ref{prop:bialg}, show that \( \H \) is a  bi-algebra.

Next, with antipode \( S \) given by \eqref{eq:antipode},
\begin{prop}\label{prop:hopf}
	For every \( p\in H \), \( v\in U(\fh) \) and \( \a\in \wedge(\fh^-) \),
	\[ S(p, v_{(1)}\otimes \a_{(1)})\ast (p, v_{(2)}\otimes \a_{(2)}) = (p, v_{(1)}\otimes \a_{(1)})\ast S(p, v_{(2)}\otimes \a_{(2)})=(e, \e(p, v\otimes \a)). \]
\end{prop}
\begin{proof}
	On one hand, 
	\begin{align*}
		S(p, v_{(1)}\otimes \a_{(1)})\ast (p, v_{(2)}\otimes \a_{(2)}) &= (p^{-1}, \A_p(S(v_{(1)}))\otimes [S(v_{(2)}),\A_p^{-1}S(\a_{(1)})]) \ast (p, v_{(3)}\otimes \a_{(2)})\\
		&=(e, S(v_{(1)})_{(1)}v_{(2)}\otimes [S(v_{(3)}), \A_p^{-1} S(\a_{(1)})] \wedge [S(v_{(1)})_{(2)}, A_p^{-1}\a_{(2)}])\\
		&=(e, [S(v), \A_p^{-1} (S(\a_{(1)})\wedge \a_{(2)})])\\
		&=(e, \e(v) \e(\a)).
	\end{align*}
	On the other hand,
	\begin{align*}
		(p, v_{(1)}\otimes \a_{(1)})\ast S(p, v_{(2)}\otimes \a_{(2)}) &= (p, v_{(1)}\otimes \a_{(1)})\ast (p^{-1}, \A_p(S(v_{(2)}))\otimes [S(v_{(3)}),\A_p^{-1}S(\a_{(2)})])\\
		&=(e, \A_p(v_{(1)}S(v_{(2)}))\otimes \a_{(1)}\wedge \A_p [ v_{(3)} S(v_{(4)}), A_p^{-1} S(\a_{(2)})])\\
		&=(e, \e(v) \e(\a)).
	\end{align*}
\end{proof}

Thus, \( \H \) is a Hopf algebra, and it can be seen from \eqref{eq:coproduct}-\eqref{eq:antipode} that the Hopf algebra operations respect the grading (and filtration.) Next, we show that \( \H \) is a differential graded Hopf algebra:

\begin{prop}\label{prop:bdry}
	The following hold for \( \p \) on \( \H \):
	\begin{enumerate}
		\item \( \p^2=0 \);
		\item \( \p \) commutes with \( \ast \), in the sense that if \( \a\in \wedge^k(\fh^-) \), then
	\[ \p((p, v\otimes \a)\ast (q, w\otimes \b))= \p(p, v\otimes \a) \ast (q, w\otimes \b) + (-1)^k(p, v\otimes \a)\ast\p(q, w\otimes \b); \]
		\item\label{third} \( \p \) commutes with \( \D \), in the sense that if \( \a\in \wedge^k(\fh^-) \), then \[ \D \p (p, v\otimes \a) = \p(p, v_{(1)}\otimes \a_{(1)})\otimes (p, v_{(2)}\otimes \a_{(2)}) + (-1)^{|\a_{(1)}|} (p, v_{(1)}\otimes \a_{(1)})\otimes \p(p, v_{(2)}\otimes \a_{(2)}); \] and
		\item \( \p \) commutes with \( S \) up to sign, in the sense that if \( \a\in \wedge^k(\fh^-) \), then \[ \p S (p, v\otimes \a) = (-1)^k S \p (p, v\otimes \a). \]
	\end{enumerate}
\end{prop}

\begin{proof}
	We proceed in order, suppressing the wedge product symbols in \( \wedge(\fh^-) \) as necessary to avoid visual clutter:
	\begin{enumerate}
		\item Suppose \( \a=\a_0\wedge\cdots \wedge \a_k \), where \( \a_i\in \fh \) for \( i=0,\dots,k \). Then \( \p\circ \p (p, v\otimes \a) \) is equal to
			\[ \p\left(p, \sum_i (-1)^i \A_p^{-1} (\a_i) v \otimes \a_0\cdots\hat{\a}_i\cdots \a_k +\sum_{i<j} (-1)^{i+j+1} v\otimes [\a_i, \a_j] \a_0\cdots \hat{\a}_i\cdots \hat{\a}_j\cdots \a_k\right), \] which is equal to \( (p, X+Y+Z) \), where
			\begin{align*}
				X&= \sum_{j<i}(-1)^{i+j} \A_p^{-1}(\a_j)\A_p^{-1}(\a_i)v\otimes \a_0\cdots\hat{\a}_j\cdots\hat{\a}_i\cdots\a_k \\
				&+ \sum_{i<j}(-1)^{i+j+1} \A_p^{-1}(\a_j)\A_p^{-1}(\a_i)v\otimes \a_0\cdots\hat{\a}_i\cdots\hat{\a}_j\cdots\a_k\\
				&+ \sum_{j<l<i}(-1)^{i+j+l+1}\A_p^{-1}(\a_i)v\otimes [\a_j,\a_l] \a_0\cdots\hat{\a}_j\cdots\hat{\a}_l\cdots\hat{\a}_i\cdots \a_k\\
				&+ \sum_{j<i<l}(-1)^{i+j+l}\A_p^{-1}(\a_i)v\otimes [\a_j,\a_l] \a_0\cdots\hat{\a}_j\cdots\hat{\a}_i\cdots\hat{\a}_l\cdots \a_k\\
				&+ \sum_{i<j<l}(-1)^{i+j+l+1}\A_p^{-1}(\a_i)v\otimes [\a_j,\a_l] \a_0\cdots\hat{\a}_i\cdots\hat{\a}_j\cdots\hat{\a}_l\cdots\a_k\\
				&+ \sum_{i<j} (-1)^{i+j+1} \A_p^{-1}([\a_i,\a_j])v \otimes \a_0\cdots\hat{\a}_i\cdots\hat{\a}_j\cdots\a_k \\
				&+ \sum_{l<i<j}(-1)^{i+j+l} \A_p^{-1}(\a_l)v \otimes [\a_i,\a_j]\a_0\cdots\hat{\a}_l\cdots\hat{\a}_i\cdots\hat{\a}_j\cdots\a_k\\
				&+ \sum_{i<l<j}(-1)^{i+j+l+1} \A_p^{-1}(\a_l)v \otimes [\a_i,\a_j]\a_0\cdots\hat{\a}_i\cdots\hat{\a}_l\cdots\hat{\a}_j\cdots\a_k\\
				&+ \sum_{i<j<l}(-1)^{i+j+l} \A_p^{-1}(\a_l)v \otimes [\a_i,\a_j]\a_0\cdots\hat{\a}_i\cdots\hat{\a}_j\cdots\hat{\a}_l\cdots\a_k,
			\end{align*}
	
			\begin{align*}
				Y&= \sum_{l<i<j}(-1)^{i+j+l} v \otimes [[\a_i,\a_j], \a_l]\a_0\cdots\hat{\a}_l\cdots\hat{\a}_i\cdots\hat{\a}_j\cdots\a_k\\
				&+ \sum_{i<l<j}(-1)^{i+j+l+1} v \otimes [[\a_i,\a_j], \a_l]\a_0\cdots\hat{\a}_i\cdots\hat{\a}_l\cdots\hat{\a}_j\cdots\a_k\\
				&+ \sum_{i<j<l}(-1)^{i+j+l} v \otimes [[\a_i,\a_j], \a_l]\a_0\cdots\hat{\a}_i\cdots\hat{\a}_j\cdots\hat{\a}_l\cdots\a_k,
			\end{align*}
			and 
			\begin{align*}
				Z&= \sum_{l<r<i<j}(-1)^{i+j+l+r} v\otimes [\a_l,\a_r][\a_i,\a_j]\a_0\cdots\hat{\a}_l\cdots\hat{\a}_r\cdots\hat{\a}_i\cdots\hat{\a}_j\cdots\a_k\\
				&+ \sum_{l<i<r<j}(-1)^{i+j+l+r+1} v\otimes [\a_l,\a_r][\a_i,\a_j]\a_0\cdots\hat{\a}_l\cdots\hat{\a}_i\cdots\hat{\a}_r\cdots\hat{\a}_j\cdots\a_k\\
				&+ \sum_{l<i<j<r}(-1)^{i+j+l+r} v\otimes [\a_l,\a_r][\a_i,\a_j]\a_0\cdots\hat{\a}_l\cdots\hat{\a}_i\cdots\hat{\a}_j\cdots\hat{\a}_r\cdots\a_k\\
				&+ \sum_{i<l<r<j}(-1)^{i+j+l+r} v\otimes [\a_l,\a_r][\a_i,\a_j]\a_0\cdots\hat{\a}_i\cdots\hat{\a}_l\cdots\hat{\a}_r\cdots\hat{\a}_j\cdots\a_k\\
				&+ \sum_{i<l<j<r}(-1)^{i+j+l+r+1} v\otimes [\a_l,\a_r][\a_i,\a_j]\a_0\cdots\hat{\a}_i\cdots\hat{\a}_l\cdots\hat{\a}_j\cdots\hat{\a}_r\cdots\a_k\\
				&+ \sum_{i<j<l<r}(-1)^{i+j+l+r} v\otimes [\a_l,\a_r][\a_i,\a_j]\a_0\cdots\hat{\a}_i\cdots\hat{\a}_j\cdots\hat{\a}_l\cdots\hat{\a}_r\cdots\a_k.
			\end{align*}
			
			Each term \( X,Y \) and \( Z \) is zero: The first, second and sixth terms of \( X \) cancel, as do the 3rd and 9th, 4th and 8th, and 5th and 7th. Likewise, the three terms of \( Y \) cancel due to the Jacobi identity. Finally, the 1st and 6th, 2nd and 5th, and 3rd and 4th terms of \( Z \) cancel. 
		\item It will be convenient to assume \( \a=\a_0\wedge\cdots\wedge \a_k\in \wedge^{k+1}(\fh) \) and \( \b=\b_0\wedge\cdots\wedge \b_l\in \wedge^{l+1}(\fh) \).
			On one hand, \( \p(p, v\otimes \a) \ast (q, w\otimes \b) = (pq, X), \) where 
			\begin{align*}
				X &= \sum_{i=0}^k (-1)^i \A_q^{-1} (\A_p^{-1} (\a_i)v_{(1)}) w \otimes \a_0\cdots\hat{\a}_i\cdots\a_k \A_p ([v_{(2)}, \b])\\
				&+ \sum_{i=0}^k (-1)^i \A_q^{-1}(v_{(1)}) w\otimes \a_0\cdots\hat{\a}_i\cdots \a_k [a_i \A_p(v_{(2)}), \A_p (\b)]\\
				&+ \sum_{0\leq i<j\leq k}(-1)^{i+j+1}\A_q^{-1}(v_{(1)})w\otimes [\a_i, \a_j]\a_0\cdots\hat{\a}_i\cdots\hat{\a}_j\dots \a_k \A_p([v_{(2)},\b]),
			\end{align*}
			and \( (p, v\otimes \a) \ast \p(q, w\otimes \b) = (pq, Y), \) where
			\begin{align*}
				Y &= \sum_{i=0}^l (-1)^i \A_q^{-1} (v_{(1)}\b_i) w\otimes \a \A_p([v_{(2)}, \b_0\cdots\hat{\b}_i\cdots \b_l])\\
				&+ \sum_{0\leq i<j\leq l}(-1)^{i+j+1}\A_q^{-1}(v_{(1)})w\otimes \a \A_p([v_{(2)}, [\b_i,\b_j]\b_0 \cdots\hat{\b}_i\cdots\hat{\b}_j\cdots \b_l]).
			\end{align*}
			On the other hand, \( \p((p, v\otimes \a) \ast (q, w\otimes \b)) \) is equal to \( (pq, Z) \), where
			\begin{align*}
				Z &= \sum_{i=0}^k (-1)^i \A_{pq}^{-1} (\a_i)\A_q^{-1}(v_{(1)}) w \otimes \a_0\cdots\hat{\a}_i\cdots\a_k \A_p ([v_{(2)}, \b])\\
				&+ \sum_{i=0}^l (-1)^{i+k+1} \A_q^{-1}([v_{(1)}, \b_i] v_{(2)}) w \otimes \a \A_p ([ v_{(3)}, \b_0\cdots\hat{\b}_i\cdots \b_l])\\
				&+ \sum_{0\leq i<j\leq k} (-1)^{i+j+1} \A_q^{-1}(v_{(1)})w\otimes [\a_i, \a_j] \a_0\cdots\hat{\a}_i\cdots\hat{\a}_j\cdots\a_k \A_p([v_{(2)}, \b])\\
				&+ \sum_{\substack{0\leq i\leq k\\ 0\leq j \leq l}}(-1)^{i+j}\A_q^{-1}(v_{(1)})w\otimes \a_0\cdots\hat{\a}_i\cdots\a_k [\a_i, \A_p(v_{(2)}),\b_j] \A_p([v_{(3)}, \b_0\cdots\hat{\b}_j\cdots\b_l])\\
				&+ \sum_{0\leq i<j\leq l}(-1)^{i+j+1} \A_q^{-1}(v_{(1)})w\otimes \left[\A_p([v_{(2)},\b_i]), \A_p ([v_{(3)}, \b_j])\right] \a \A_p([v_{(4)}. \b_0\cdots\hat{\b}_i\cdots\hat{\b}_j\cdots\b_l].
			\end{align*}
			
			We observe that \( X+(-1)^{k+1}Y = Z \), noting that the second term of \( Z \) simplifies to \[ \sum_{i=0}^l (-1)^{i+k+1} \A_q^{-1}(v_{(1)}\b_i) w \otimes \a \A_p ([ v_{(2)}, \b_0\cdots\hat{\b}_i\cdots \b_l]), \] since if \( x\in U(\fh) \) and \( y\in \fh \), the commutation relation \( [x_{(1)},y]x_{(2)}=xy \) holds in \( U(\fh) \).
			
		\item On one hand, \( \p \D (p, v\otimes \a_0\cdots \a_k) \) is equal to
			\begin{align*}
				&\quad \sum_{\s} \sgn(\s) \p(p, v_{(1)}\otimes \a_{\s(0)}\cdots \a_{\s(l)}) \otimes (p, v_{(2)}\otimes \a_{\s(l+1)}\cdots \a_{\s(k)})\\
				&+ \sum_{\s} \sgn(\s)(-1)^{l+1} (p, v_{(1)}\otimes \a_{\s(0)}\cdots \a_{\s(l)}) \otimes \p(p, v_{(2)}\otimes \a_{\s(l+1)}\cdots \a_{\s(k)}),
			\end{align*}
			where the sums are taken over all \( (l+1,k+1) \)-shuffle permutations. This expands to
			\begin{align*}
				&\quad\sum_\s \sgn(\s)\sum_{i=0}^l(-1)^i(p, \A_p^{-1}(\a_{\s(i)})v_{(1)}\otimes \a_{\s(0)}\cdots\hat{\a}_{\s(i)}\cdots\a_{\s(l)})\otimes (p, v_{(2)}\otimes \a_{\s(l+1)}\cdots\a_{\s(k)})\\
				&+ \sum_\s \sgn(\s)\sum_{0\leq i<j\leq l} (-1)^{i+j+1}(p, v_{(1)}\otimes [\a_{\s(i)},\a_{\s(j)}] \a_{\s(0)}\cdots\hat{\a}_{\s(i)}\cdots\hat{\a}_{\s(j)}\cdots\a_{\s(l)})\otimes (p, v_{(2)}\otimes \a_{\s(l+1)}\cdots \a_{\s(k)})\\
				&+ \sum_\s \sgn(\s) \sum_{i=l+1}^k (-1)^{i}(p, v_{(1)}\otimes \a_{\s(0)}\cdots\a_{\s(l)})\otimes(p, \A_p^{-1}(\a_{\s(i)})v_{(2)}\otimes \a_{\s(l+1)}\cdots \hat{\a}_{\s(i)}\cdots \a_{\s(k)})\\
				&+ \sum_\s \sgn(\s)\sum_{l+1\leq i<j\leq k}(-1)^{i+j+l}(p, v_{(1)}\otimes \a_{\s(0)}\cdots\a_{\s(l)})\otimes (p, v_{(2)}\otimes[\a_{\s(i)},\a_{\s(j)}]\a_{\s(l+1)}\cdots\hat{\a}_{\s(i)}\cdots \hat{\a}_{\s(j)}\cdots\a_{\s(k)}).
			\end{align*}
			On the other hand, \( \D \p (p, v\otimes \a_0\wedge\cdots \wedge \a_k) \) equals
		
			\[ \D\left(p,\sum_{i=0}^k(-1)^i (\Ad_p^{-1}(\a_i)) v \otimes \a_0\wedge\cdots \hat{\a_i}\cdots\wedge \a_k + \sum_{i<j}(-1)^{i+j+1}v\otimes [\a_i,\a_j]\wedge\a_1\wedge\cdots \hat{\a_i}\cdots\hat{\a_j}\cdots\wedge\a_k\right), \]
			which expands to
			\begin{align*}
				&\quad \sum_{\substack{\s\in\sh_i(l+1,k+1)\\ 0\leq i \leq k}} (-1)^{\s^{-1}(i)} \sgn(\s) (p, \A_p^{-1}(\a_i)v_{(1)}\otimes \a_{\s(0)}\cdots\hat{\a}_i\cdots \a_{\s(l)})\otimes (p v_{(2)}\otimes \a_{\s(l+1)}\cdots \a_{\s(k)})\\
				&+ \sum_{\substack{\s\in \sh^i(l+1,k+1)\\ 0\leq i \leq k}}(-1)^{\s^{-1}(i)}\sgn(\s)(p, v_{(1)}\otimes \a_{\s(0)}\cdots \a_{\s(l)})\otimes (p, \A_p^{-1}(\a_i) v_{(2)}\otimes \a_{\s(l+1)}\cdots\hat{\a}_i\cdots \a_{\s(k)})\\
				&+ \sum_{\substack{\s\in\sh_{i,j}(l+1,k+1)\\ 0\leq i<j\leq k}}(-1)^{\s^{-1}(i)+\s^{-1}(j)+1}\sgn(\s)(p, v_{(1)}\otimes [\a_i,\a_j]\a_{\s(0)}\cdots \hat{\a}_i\cdots\hat{\a}_j\cdots\a_{\s(l)})\otimes(p, v_{(2)}\otimes \a_{\s(l+1)}\cdots\a_{\s(k)})\\
				&+ \sum_{\substack{\s\in\sh^{i,j}(l+1,k+1)\\ 0\leq i<j\leq k}}(-1)^{\s^{-1}(i)+\s^{-1}(j)+l}\sgn(\s)(p, v_{(1)}\otimes \a_{\s(0)}\cdots \a_{\s(l)})\otimes (p, v_{(2)}\otimes [\a_i,\a_j]\a_{\s(l+1)}\cdots\hat{\a}_i\cdots\hat{\a}_j\cdots \a_{\s(k)}),
			\end{align*}
			where for \( L=i \) and \( L=i,j \), the set \( \sh_L(l+1,k+1) \) (resp. \( \sh^L(l+1,k+1) \)) consists of those shuffle permutations such that \( \{L\}\subseteq \{\s(0),\dots,\s(l)\} \) (resp.  \( \{L\}\subseteq \{\s(l+1),\dots,\s(k)\} \)). This sum is equal to \( \p \D (p, v\otimes \a_0\cdots \a_k) \) by a change of variables.
			
			\item Suppose \( \a=\a_0\cdots \a_k \). On one hand, \( \p S(p, v\otimes \a) \) is equal to \( (p^{-1}, A+ B) \), where 
			\begin{align*}
				A&= \sum_{i=0}^k (-1)^{i+1}[\A_p S(v_{(1)}), \a_i]\A_p (S(v_{(2)}))\otimes [S(v_{(3)}), \A_p^{-1}(\a_k\cdots\hat{\a}_i\cdots \a_0)]\\
				&= \sum_{i=0}^k (-1)^{i+1}\left(\A_p S(v_{(1)})\right)\a_i \otimes \left[ S(v_{(2)}), \A_p^{-1}(\a_k\cdots\hat{\a}_i\cdots \a_0)\right]
			\end{align*}
			and
			\begin{align*}
				B&= \sum_{0\leq i<j\leq k} (-1)^{i+j+k} \A_p(S(v_{(1)}))\otimes \left[ \left[S(v_{(2)}),\A_p^{-1}(\a_j)\right], \left[S(v_{(3)}), \A_p^{-1}(\a_i)\right]\right]\left[S(v_{(4)}), \A_p^{-1}(\a_k\cdots\hat{\a}_j\cdots\hat{\a}_i\cdots \a_0) \right]\\
				&= \sum_{0\leq i<j\leq k} (-1)^{i+j+k} \A_p(S(v_{(1)}))\otimes\left[S(v_{(2)}), \A_p^{-1}[\a_j,\a_i]\a_k\cdots\hat{\a}_j\cdots\hat{\a}_i\cdots \a_0 \right].
			\end{align*}
			On the other hand, \( S \p(p, v\otimes \a) \) is equal to \( (p^{-1}, X+Y) \), where 
			\begin{align*}
				X &= \sum_{i=0}^k (-1)^{i+k+1} \A_p (S((\A_p^{-1}(\a_i)v)_{(1)}))\otimes \left[ S((A_p^{-1}(\a_i)v)_{(2)}), \A_p^{-1}\left(\a_k\cdots\hat{\a}_i\cdots \a_0\right) \right]\\
				&= \sum_{i=0}^k (-1)^{i+k+1} \A_p(S(v_{(1)})) \a_i \otimes \left[S(v_{(2)}),\A_p^{-1}\left(\a_k\cdots \hat{\a}_i\cdots \a_0\right) \right]\\
				&\quad + \sum_{i=0}^k (-1)^{i+k+1} \A_p(S(v_{(1)})) \otimes \left[S(v_{(2)})\A_p^{-1}(\a_i),\A_p^{-1}\left(\a_k\cdots \hat{\a}_i\cdots \a_0\right) \right]\\
				&= \sum_{i=0}^k (-1)^{i+k+1} \A_p(S(v_{(1)})) \a_i \otimes \left[S(v_{(2)}),\A_p^{-1}\left(\a_k\cdots \hat{\a}_i\cdots \a_0\right) \right]\\
				&\quad + 2\sum_{0\leq i<j\leq k}(-1)^{i+j}\A_p(S(v_{(1)}))\otimes\left[S(v_{(2)}),\A_p^{-1}([\a_j,\a_i]\a_k\cdots\hat{\a}_j\cdots\hat{\a}_i\cdots \a_0) \right],
			\end{align*}
			and
			\begin{align*}
				Y &= \sum_{0\leq i<j\leq k}(-1)^{i+j+1}\A_p(S(v_{(1)}))\otimes\left[S(v_{(2)}),\A_p^{-1}S([\a_i,\a_j]\a_0\cdots\hat{\a}_i\cdots\hat{\a}_j\cdots \a_k) \right]\\
				&= \sum_{0\leq i<j\leq k}(-1)^{i+j+1}\A_p(S(v_{(1)}))\otimes\left[S(v_{(2)}),\A_p^{-1}([\a_j,\a_i]\a_k\cdots\hat{\a}_j\cdots\hat{\a}_i\cdots \a_0) \right],
			\end{align*}
			and the result follows.
			
	\end{enumerate}
\end{proof}

\begin{proof}[Proof of Theorem \ref{thm:hopf4}]
	 By examining the formulas, we see that equations \eqref{eq:boundary}-\eqref{eq:antipode} define continuous (resp. smooth) bundle maps if \( H \) is a topological (resp. Lie) group and \( \A \) is continuous. These bundle maps lift to bundle maps on \( H\times(\otimes(\fh)\otimes \wedge(\fh^-)) \), using the same formulas as \eqref{eq:boundary}-\eqref{eq:antipode}, with multiplication on the left of the tensor product occurring in \( \otimes(\fg) \) instead of \( U(\fg) \). It is enough to run through the previous section replacing products in \( U(\fg) \) with products in \( \otimes(\fg) \). The required equalities still hold, except in Proposition \ref{prop:bdry}, in which the only part that still holds is part \ref{third}.
\end{proof}

\addcontentsline{toc}{section}{References} 
\bibliography{Harrisonbib.bib}{}
\bibliographystyle{amsalpha}
\end{document}